\documentclass[12pt]{article}
\usepackage{amsfonts,amssymb,amsmath,titling,titlesec,cancel,amsthm,color, hyperref}
\usepackage[numbers,square]{natbib}

\titleformat{\section}{\large\bfseries}{\thesection.}{1em}{}

\theoremstyle{plain}
\newtheorem{theorem}{Theorem}[section]
\newtheorem{lemma}[theorem]{Lemma}

\newtheorem{corollary}[theorem]{Corollary}
\newtheorem{definition}[theorem]{Definition}

\newtheorem{conjecture}[theorem]{Conjecture}

\title{Resolving a Conjecture on Degree of Regularity of Linear Homogeneous Equations}
\date{\today} 
\author{Noah Golowich\thanks{MIT-PRIMES, Department of Mathematics, Massachusetts Institute of Technology.}}

\begin{document}
\maketitle

\begin{abstract}
A linear equation is {\it $r$-regular}, if, for every $r$-coloring of the positive integers, there exist positive integers of the same color which satisfy the equation. In 2005, Fox and Radoi\' ci\v c conjectured that the equation $x_1 + 2x_2 + \cdots + 2^{n-2}x_{n-1} - 2^{n-1}x_n = 0$, for any $n \geq 2$, has a degree of regularity of $n-1$, which would verify a conjecture of Rado from 1933. Rado's conjecture has since been verified with a different family of equations. In this paper, we show that Fox and Radoi\' ci\v c's family of equations indeed have a degree of regularity of $n-1$. We also provide a few extensions of this result.
\end{abstract}

\section{Introduction}
\label{sec:intro}
In 1927, van der Waerden \cite{waerden_beweis_1927} proved his seminal theorem stating that, for any finite coloring of the positive integers, there always exists a monochromatic arithmetic progression of arbitrary length. Subsequently, in 1933, Rado \cite{rado_studien_1933} expanded on this theorem, finding a necessary and sufficient condition for the partition regularity of systems of linear homogeneous equations. In the case of a single linear homogeneous equation of the form
\begin{equation}
\label{eq:initialeq}
a_1x_1 + a_2x_2 +a_3x_3 + \cdots + a_nx_n = 0 \quad : \quad a_i \ne 0, \ a_i \in \mathbb{Z},
\end{equation}
given any finite coloring of the integers, Rado proved that there exist positive monochromatic integers $(x_1, x_2, \ldots, x_n)$ that satisfy the equation if and only if a nonempty subset of $\{a_1, a_2, \ldots, a_n\}$ sums to 0. An equation for which there exists a monochromatic solution given any finite coloring is defined as {\it regular}.

Not all linear homogeneous equations are regular, however. Those that are not regular are classified as follows: given a positive integer $r$, a linear homogeneous equation is called {\it $r$-regular} if, for every coloring of the positive integers with $r$ colors, there always exists a monochromatic solution $x_1, x_2, \ldots, x_n$ to the equation. The {\it degree of regularity} of a linear homogeneous equation is defined as the largest positive integer $r$ such that the equation is $r$-regular.

Rado conjectured in 1933 \cite{rado_studien_1933} that for every positive integer $n$, there exists a linear homogeneous equation with degree of regularity equal to $n$. This conjecture was open for a long time until it was proven in 2009 by Alexeev and Tsimerman \cite{alexeev_equations_2010}. Specifically, they proved that for each $n$ the equation
\begin{equation}
\label{eq:alexeev}
 \left(1 - \displaystyle\sum\limits_{i = 1}^{n-1} \frac{2^i}{2^i - 1}\right)x_1 + \displaystyle\sum\limits_{i = 1}^{n-1} \frac{2^i}{2^i-1}x_{i+1}  = 0
\end{equation}
is $(n-1)$-regular but not $n$-regular. To show that these equations are $(n-1)$-regular, they noted that there must be an $i$, $0 < i < n$, such that $x$ and $2^i x$ are the same color. Otherwise, the $n$ integers $x, 2x, 4x, \ldots 2^{n-1}x$ would all be  different colors, which is impossible in an $(n-1)$-coloring. Alexeev and Tsimerman then noted that the following is a monochromatic solution to $(\ref{eq:alexeev})$:
\begin{eqnarray}
x_1 = x_2 = \cdots = x_{i} = x_{i+2} = \cdots = x_n &=& 2^ix\nonumber\\
x_{i+1} &=& x\nonumber.
\end{eqnarray}

Before Alexeev and Tsimerman's proof, it was conjectured in 2005 by Fox and Radoi\' ci\v c \cite{fox_axiom_2005} that the simpler family of equations
\begin{equation}
\label{eq:altsintro}
x_1 + 2x_2 + \cdots + 2^{n-2}x_{n-1} - 2^{n-1}x_n = 0
\end{equation}
is $(n-1)$-regular. This was shown by Alexeev, Fox, and Graham for $n \leq 7$ \cite{alexeev_minimal_2007}. We prove this conjecture:

\begin{theorem}
\label{thm:main}
Given any positive integer $n \geq 2$, the equation 
\begin{equation}
\label{eq:2ieqn}
x_1 + 2x_2 + \cdots + 2^{n-2}x_{n-1} - 2^{n-1}x_n = 0
\end{equation}
is $(n-1)$-regular.
\end{theorem}

Fox and Radoi\' ci\v c showed that these equations are not $n$-regular \cite{fox_axiom_2005}. Our proof therefore provides an alternate proof of Rado's conjecture.

To prove that $(\ref{eq:2ieqn})$ is $(n-1)$-regular, we use a similar strategy to Alexeev and Tsimerman. Specifically, we use the fact that in an $(n-1)$-coloring of the positive integers, one can find $0 < i < n$ such that $x$ and $2^ix$ are the same color. We then introduce some lemmas which use van der Waerden's theorem to prove that one can find monochromatic progressions centered at $x$ and $2^ix$. Finally, we find a monochromatic parametrization of $x_1, \ldots, x_n$ which includes the progressions centered at $x$ and $2^ix$.

The organization of this paper is as follows: In Section $\ref{sec:lemmas}$ we prove some important lemmas. In Section $\ref{sec:mainthm}$, we prove our main result. In Section $\ref{sec:cors}$, we prove some extensions of Theorem $\ref{thm:main}$.

\section{Lemmas}
\label{sec:lemmas}
The below lemmas are very similar to lemmas in Graham, Rothschild, and Spencer \cite{graham_ramsey_1990} used in the proof of Rado's Theorem. The main difference is that the hypotheses are relaxed to a family of sets being $r$-regular instead of regular, and as a result, the lemmas apply to an $r$-coloring of the positive integers, and not any finite coloring. Similar versions of these lemmas have also been used to strengthen the proof of Rado's conjecture \cite{alexeev_equations_2010}\cite{gandhi_degree_2013}. A few definitions are needed beforehand.

\begin{definition}
A family $S$ of subsets of $\mathbb{N}$ is defined to be {\bf homogeneous}, if for any set $A \in S$, the set $A'$, defined as $\{ka : a \in A\}$ also belongs to S. $S$ is {\bf $r$-regular}, if for any $r$-coloring of the positive integers, there exists a monochromatic set in $S$.
\end{definition}

\begin{lemma} 
\label{lem:manyprogs}
Let $S$ be a homogeneous family of subsets of $\mathbb{N}$ which is $r$-regular, where $r$ is a positive integer. If the positive integers  are colored with $r$ colors, then there exist $B \in S$, $d, M > 0$, so that all
\begin{equation}
b + \lambda d \quad : \quad b \in B, |\lambda| \leq M\nonumber
\end{equation}
have the same color.
\end{lemma}

\begin{proof}
By the Compactness principle \cite{graham_ramsey_1990}, we can find a constant $R$ such that for any $r$-coloring of $[1, R]$, there exists a monochromatic set in $S$ which is a subset of $[1, R]$.  Now, let $\omega$ be an $r$-coloring of $\mathbb{N}$.  We define an $r^{R}$ coloring on $\mathbb{N}$, namely $\omega'$, by 
\[\omega'(\alpha) = \omega'(\beta) \ \ \ \ \text{iff} \ \ \ \ \omega(\alpha i) = \omega(\beta i) \mbox{ for } 1 \le i \le R. \]
We define $K = CR^{n-1}$. Within our coloring $\omega'$, we find a monochromatic arithmetic progression $P$ of length $2K + 1$ of the form 
\begin{equation}
\label{eq:othermonoseq}
a + \lambda d \quad : \quad |\lambda| \le K.
\end{equation}
Note that for the original coloring $\omega$, this means that $P$, $2\cdot P$,\ldots, $R\cdot P$ are all monochromatic.  Since $S$ is homogeneous, the $r$-coloring $\omega$ of $\{a, 2a, \ldots, Ra\}$ yields a monochromatic set in $S$, namely $B$. 

Now we let $n$ be the size of $B$ and let $ab_1, \ldots,a b_n$ be the elements of $B$. Moreover, we let $y$ be the least common multiple of $b_1, \ldots, b_n$ and $d' = d y$. Therefore,
\[ab_i + \lambda d'  = ab_i + \lambda d y = b_i\left(a + \lambda d\left(\frac{y}{b_i}\right)\right).\]
Notice that $\lambda \le C$ and $\frac{y}{b_i} \le b_1\ldots b_{i-1} b_{i+1} \ldots b_n \le R^{n-1}$, so $\lambda y/ b_i \le K$.  

Thus, $a + \lambda d y/b_i$ belongs to $P$, which implies that 
\[\omega'\left(a + \lambda d \left(\frac{y}{b_i}\right)\right) = \omega'(a),\]
and thus, by our definition of $\omega'$,
\begin{equation}
\omega(ab_i + \lambda d') = \omega(ab_i).\nonumber
\end{equation}
Since $\omega(ab_i)$ is constant for $1 \leq i \leq n$, $\omega(ab_i + \lambda d')$ is constant for $1 \leq i \leq n$ and $|\lambda| \leq C$.
\end{proof}

\begin{corollary}
\label{cor:progd}
Let $S$ be a homogeneous family of subsets of $\mathbb{N}$ which is $r$-regular, where $r$ is a positive integer. Let $q$ and $M$ be positive integers. If the positive integers are colored with $r$ colors, then there exist $B \in S$, $d > 0$ such that all of
\begin{equation}
\label{eq:firstcon}
b + \lambda d \quad : \quad b \in B, |\lambda| \leq M
\end{equation}
and
\begin{equation}
\label{eq:secondcon}
qd
\end{equation}
have the same color.
\end{corollary}

\begin{proof}
We use induction on the number of colors. Let $p$ be a positive integer less than or equal to $r$. Assume that there exists $T = T(p-1, M, c)$ so that if $[1, T]$ is $(p-1)$-colored, there exist $B \in S$ and $d$ satisfying $(\ref{eq:firstcon})$ and $(\ref{eq:secondcon})$. Now consider any $p$-coloring of the positive integers. By Lemma $\ref{lem:manyprogs}$, and since $p \leq r$, there exist $B \in S$ and $d' > 0$ such that all
\begin{equation}
b + \lambda d' \quad : \quad b \in B, |\lambda| \leq TM\nonumber
\end{equation}
are the same color, suppose red. If there exists an integer $u \leq T$ where $uqd'$ is also colored red then we let $d = ud'$, which would satisfy $(\ref{eq:firstcon})$ and $(\ref{eq:secondcon})$. Otherwise, the integers $qd', 2qd', \ldots, Tqd'$ are each colored with one of $p-1$ colors. Therefore, $(\ref{eq:firstcon})$ and $(\ref{eq:secondcon})$ are satisfied by induction.
\end{proof}

\section{Proof of Main Theorem}
\label{sec:mainthm}
In this section we prove Theorem $\ref{thm:main}$.

\begin{proof}
Consider any $(n-1)$-coloring $c: \mathbb{N} \rightarrow \{1, \ldots, n-1\}$ of the positive integers. By the pigeonhole principle, among the set of integers $\{2^0, 2^1, \ldots, 2^{n-1}\}$, there must be 2 integers of the same color. Therefore, among the first $2^{n-1}$ positive integers, there are integers $j$, $1 \leq j \leq n-1$, and $x$, such that $c(x) = c(2^jx)$. We now define the set $S$, a family of ordered pairs by:
\begin{equation}
S = \left\{(a, b) \in \mathbb{N}^2 \quad : \quad a = 2^jb \mbox{ where } 1 \leq j \leq n-1\right\}.\nonumber
\end{equation}
Notice that $S$ is  homogeneous, Moreover, by the argument above, $S$ is $(n-1)$-regular on $\mathbb{N}$; specifically, a monochromatic set in $S$ can be found on any $(n-1)$-coloring of the first $2^{n-1}$ integers. We now let $M = 2^n$.

By Corollary $\ref{cor:progd}$, for any $(n-1)$-coloring of $\mathbb{N}$ there exist $(a, b) \in S$, $d > 0$, where all
\begin{eqnarray}
&a + l_1d& \quad : \quad |l_1| \leq M \nonumber\\
&b + l_2d& \quad : \quad |l_2| \leq M\nonumber\\
&2^{n-1}d&\nonumber
\end{eqnarray}
 are the same color. By the definition of $S$, there exists a positive integer $j$ such that $a = 2^j b$. We now form the following parametrization of $x_1, \ldots, x_n$:
\begin{eqnarray}
x_1 &=& 2^{n-1}d\nonumber\\
&\vdots& \nonumber\\
x_{n-j} &=& 2^jb + \lambda_1d\nonumber\\
\label{eq:xiparam}
&\vdots&\\
x_{n-1} &=& 2^{n-1}d\nonumber\\
x_n &=& b + \lambda_2d\nonumber,
\end{eqnarray}
where $\lambda_1$ and $\lambda_2$ are integers in the range $-M \leq \lambda_1, \lambda_2 \leq M$, to be determined. Clearly, each of the values assigned to $x_1, \ldots, x_n$ is the same color, regardless of the values of $\lambda_1, \lambda_2$. In order for the above values of $x_1, \ldots, x_n$ to satisfy $(\ref{eq:2ieqn})$, we must have:
\begin{equation}
2^{n-1}d(2^{n-1} - 1 - 2^{n - 1 - j}) + 2^{n - 1 - j}(2^jb + \lambda_1d) - 2^{n-1}(b + \lambda_2d) = 0\nonumber,
\end{equation}
which implies that
\begin{equation}
2^{2n-2} - 2^{n-1} - 2^{2n - 2 - j} + 2^{n - 1 - j}\lambda_1 - 2^{n-1}\lambda_2 = 0.\nonumber
\end{equation}
Rearranging terms and dividing out by $2^{n -1 - j}$, we see that
\begin{equation}
\lambda_1 - 2^j \lambda_2 = 2^{n-1} + 2^j - 2^{n - 1 + j}.\nonumber
\end{equation}
We may now choose $\lambda_2 = 2^{n-1}$ and $\lambda_1 = 2^{n-1} + 2^{j}$, both of which are less than or equal to $M$. These values of $\lambda_1, \lambda_2$ produce $x_1, \ldots, x_n$ according to $(\ref{eq:xiparam})$ which are both monochromatic and satisfy $(\ref{eq:2ieqn})$.
\end{proof}

\section{Extensions}
\label{sec:cors}

We can derive several extensions of our results. The first states that adding additional terms to a linear homogeneous equation can not lower its degree of regularity.

\begin{theorem}
Assume the equation $a_1x_1 + \cdots + a_nx_n = 0$ is $r$-regular. Then, for any positive integer $k$, and rationals $b_1, \ldots, b_k$, the equation
\begin{equation}
\label{eq:aandb}
a_1x_1 + \cdots + a_nx_n + b_1x_{n+1} + \cdots + b_kx_{n+k} = 0
\end{equation}
is also $r$-regular.
\end{theorem}
\begin{proof}
Given any $r$-coloring of the integers, we can find $(x_1, \ldots, x_n)$ of the same color which satisfy $\sum_{i = 0}^n a_ix_i = 0$.  By Corollary $\ref{cor:progd}$, for any positive integer $M$, there are $y_1, \ldots, y_n, d > 0$ where $\sum_{i = 0}^n a_iy_i = 0$ and such that all of:
\begin{equation}
y_i + \lambda d \quad : \quad |\lambda| \leq M\nonumber
\end{equation}
and 
\begin{equation}
a_1d
\end{equation}
have the same color. We now consider the monochromatic parametrization:
\begin{eqnarray}
x_1 &=& y_1 + \lambda_1 d\nonumber\\
&\vdots& \nonumber\\
x_{n} &=& y_n + \lambda_nd\nonumber\\
x_{n+1} &=&a_1d\nonumber\\
&\vdots&\nonumber\\
x_{n+k} &=& a_1d\nonumber,
\end{eqnarray}
which is a solution to $(\ref{eq:aandb})$ if and only if:
\begin{equation}
a_1\lambda_1 + \cdots + a_n\lambda_n + (b_1 + \cdots + b_k)a_1 = 0\nonumber.
\end{equation}
The above equation is satisfied if we let $\lambda_1 = -(b_1 + \cdots + b_k)$, $\lambda_2 = \cdots = \lambda_n = 0$.
\end{proof}

We next prove that given any linear homogeneous equation $E$ of $n$ variables, for any $(n-1)$-coloring of the positive integers, there exist monochromatic integers which satisfy $E$ with at most one of the signs of the coefficients changed. This is equivalent to stating that for an $(n-1)$-coloring of the positive integers, a monochromatic solution will always be found on one of the $n+1$ hyperplanes defined by changing exactly 1 or 0 signs of the coefficients of $E$.

\begin{theorem}
\label{thm:hyperplanes}
Suppose we are given the equation $\sum_{i = 1}^n a_ix_i = 0$. Then for any $(n-1)$-coloring of the positive integers, there exist $x_1, \ldots, x_n$ and a function $f: \{1, 2, \ldots, n\} \rightarrow \{-1, 1\}$, where there is at most one $i$ such that $f(i) = -1$, such that
\begin{equation}
\label{eq:signschanged}
f(1)a_1x_1 + \cdots + f(n)a_nx_n = 0.
\end{equation}
\end{theorem}
\begin{proof}
We define the family of sets 
\begin{equation}
S : \{ (|a_i|k, |a_j|k) \quad : \quad k \in \mathbb{N}, \mbox{  }1 \leq i < j \leq n\}.\nonumber
\end{equation}
 $S$ is clearly homogeneous. We claim that $S$ is also $(n-1)$-regular. To show this, consider any $(n-1)$-coloring of the positive integers. Then, by the pigeonhole principle, of the $n$ integers $|a_1|, \ldots |a_n|$, two must have the same color. The resulting set is monochromatic and belongs to $S$. We define $P = |a_1 \cdots a_n|$.
 
 By Corollary $\ref{cor:progd}$, for any $(n-1)$-coloring of $\mathbb{N}$, and for any positive integer $M$, there exist $\{|a_i|k, |a_j|k\} \in S, d >0$, where all
\begin{eqnarray}
&|a_i|k + l_1d& \quad : \quad |l_1| \leq M \nonumber\\
&|a_j|k + l_2d& \quad : \quad |l_2| \leq M\nonumber\\
&Pd&\nonumber
\end{eqnarray}
have the same color. 

If $a_i$ and $a_j$ have opposite signs, then we define $f(1) = \cdots = f(n) = 1$. Otherwise, we define $f(1) = \cdots = f(i - 1) = f(i+1) = \cdots = f(n) = 1$, $f(i) = -1$.

We now proceed almost identically as to in Theorem $\ref{thm:main}$. We will prove that $(\ref{eq:signschanged})$ has a monochromatic solution. We consider the following monochromatic parametrization:
\begin{eqnarray}
x_1 &=& Pd\nonumber\\
x_2 &=& Pd\nonumber\\
&\vdots& \nonumber\\
x_{i} &=& |a_j|k + \lambda_1d\nonumber\\
&\vdots&\nonumber\\
x_{j} &=& |a_i|k + \lambda_2d\nonumber\\
&\vdots&\nonumber\\
x_n &=& Pd\nonumber,
\end{eqnarray}
The above $x_i$ satisfy $(\ref{eq:signschanged})$ if and only if 
\begin{equation}
f(i)a_i\lambda_1 + a_j\lambda_2 + P(a_1 + \cdots + a_{i-1} + a_{i+1} + \cdots + a_{j-1} + a_{j+1} + \cdots + a_n) = 0\nonumber.
\end{equation}
We now choose $\lambda_2 = 0$, 
\begin{equation}
\lambda_1 = - \frac{P(a_1 + \cdots + a_{i-1} + a_{i+1} + \cdots + a_{j-1} + a_{j+1} + \cdots + a_n)}{f(i)a_i},\nonumber
\end{equation}
which is an integer, by the definition of $P$.
\end{proof}

\section{Conclusion}
Both Theorem \ref{thm:main} and the result of Alexeev and Tsimerman \cite{alexeev_equations_2010} show that certain equations of $n$ variables are $(n-1)$-regular. Rado \cite{rado_studien_1933} conjectured that the degree of regularity of any linear homogeneous equation with $n$ variables which is not regular is bounded above by some function of $n$. Fox and Kleitman proved this conjecture for $n = 3$. \cite{fox_rados_2006} It seems that there is also a nontrivial lower bound on the degree of regularity for most linear homogeneous equations:
\begin{conjecture}
\label{conj:final}
For each positive integer $r$ there is an integer $n(r)$ such that for any $n \geq n(r)$, any linear homogeneous equation in $n$ variables with nonzero integer coefficients not all of the same sign is $r$-regular.
\end{conjecture}
The requirement that the linear homogeneous equation have coefficients not all of the same sign is necessary even if we allow solutions to belong to the nonzero integers. For instance, given the equation $x_1 + \cdots + x_n = 0$, we can color the positive integers blue and the negative integers red to exclude monochromatic solutions, meaning that the equation's degree of regularity is 1 for any $n$.

\section{Acknowledgements}
The author would like to thank L\' aszl\' o Mikl\' os Lov\' asz for his suggestions on the paper and helpful conversations, and Professor Jacob Fox for introducing him to this field and for helpful discussions. 
Also, the author would like to acknowledge the MIT-PRIMES program for providing the opportunity to perform this research.



\bibliographystyle{ieeetr}
\bibliography{arxiv_note_library2.bib}
\end{document}